\documentclass[11pt,reqno]{amsart}
\usepackage[T1]{fontenc}
\usepackage{lmodern}
\usepackage{amsmath,amssymb,amsthm}
\usepackage[textheight=8.85in,textwidth=6.75in]{geometry}
\usepackage{microtype}
\usepackage[hidelinks]{hyperref}
\hypersetup{
 pdftitle={Almost All Zeros of Dedekind Jensen Polynomials Are Real},
 pdfauthor={Ken Ono},
 pdfsubject={Dedekind Jensen polynomials and their real zeros},
 pdfkeywords={Dedekind zeta-functions, Generalized Riemann hypothesis, Jensen polynomials}
}
\numberwithin{equation}{section}
\newtheorem{theorem}{Theorem}
\newtheorem{proposition}[theorem]{Proposition}
\newtheorem{remark}[theorem]{Remark}
\newtheorem{lemma}[theorem]{Lemma}
\newtheorem{corollary}[theorem]{Corollary}
\DeclareMathOperator{\Log}{Log}
\DeclareMathOperator{\Arg}{Arg}
\renewcommand{\Re}{\operatorname{Re}}
\renewcommand{\Im}{\operatorname{Im}}
\newcommand{\R}{\mathbb R}
\newcommand{\C}{\mathbb C}
\newcommand{\dd}{\,d}

\title[An Asymptotic Jensen--P\'olya Criterion for GRH]{An Asymptotic Jensen--P\'olya Criterion for GRH}
\author{Ken Ono}
\address{Axiom Math\\124 University Avenue\\Palo Alto, CA 94301}
\email{ken@axiommath.ai}

\begin{document}
\begin{abstract}
The Jensen--P\'olya criterion expresses the generalized Riemann
hypothesis (GRH) for Dedekind zeta-functions as a real-rootedness
condition on their Jensen polynomials. Previous asymptotic results
require large shifts, while computations establish real-rootedness
in the Riemann case through a finite range of degrees.
Here we study the unshifted family as the degree tends to infinity.
For every number field $K$, we prove unconditionally that
$J_{K,d}(X)$ is nonzero and has at least
$d-C_K\sqrt{d\log d}$ distinct negative real zeros for all
sufficiently large $d$, where $C_K>0$ depends only on $K$.
Therefore, the proportion of real zeros tends to one with error
$O_K(\sqrt{\log d/d})$. 
The main theorems have been autoformalized in Lean~4 by AxiomProver.
\end{abstract}

\maketitle

\section{Introduction}\label{sec:introduction}

The Riemann Hypothesis (RH) asserts that every nontrivial zero of the
Riemann zeta function $\zeta(s)$ has real part $1/2$. Here $\zeta(s)$
is the meromorphic continuation of $\sum_{m\ge1}m^{-s}$ from
$\Re s>1$, and its nontrivial zeros are those in the strip
$0<\Re s<1$. A classical program originating with Jensen \cite{Jensen} and
P\'olya \cite{Polya} studies RH through the zeros of real
polynomials. Its starting point is the entire xi-function
\begin{equation}\label{eq:xi-definition}
 \xi(s)=\frac12s(s-1)\pi^{-s/2}\Gamma(s/2)\zeta(s),
\end{equation}
where $\Gamma$ is Euler's gamma function. The zeros of $\xi$ are
precisely the nontrivial zeros of $\zeta$, with the same multiplicities.
Since the functional equation $\xi(s)=\xi(1-s)$ makes
$\xi(1/2+w)$ even in $w$, we can collect its even Taylor coefficients
in the entire function
\begin{equation}\label{eq:F-definition}
 F(z)=\xi\!\left(\frac12+\sqrt z\right)
     =\sum_{j=0}^{\infty}\frac{\gamma(j)}{j!}z^j.
\end{equation}
The value of $F$ is independent of the choice of square root, and its
coefficients $\gamma(j)$ are positive, as we recall in
Section~\ref{sec:preliminaries}.

For integers $d\ge1$ and $n\ge0$,
we then define the degree $d$ Jensen polynomial with
shift $n$ by
\begin{equation}\label{eq:shifted-J}
 J^{d,n}(X)=\sum_{j=0}^{d}\binom dj\gamma(n+j)X^j.
\end{equation}
A real polynomial is called \emph{hyperbolic} if all its zeros are real.
Griffin, Rolen, Zagier, and the author \cite{GORZ} proved that,
for each fixed degree $d$, the polynomial $J^{d,n}(X)$ is
hyperbolic for all sufficiently large shifts $n$. Their proof
identifies Hermite polynomials, which are known to be hyperbolic, as limits after suitable
normalization. This result was made effective in
\cite{GORTTW}. The present best result is due to Kim and Lee \cite{KL}. 
For every $\varepsilon>0$, they establish hyperbolicity for
$n\gg_{\varepsilon} d^{1/2+\varepsilon}$. In these results, the bounds allow degree and shift
to grow together, but require the shift to tend to infinity.
Wagner \cite[Theorem~1.3 and Corollary~1.6]{Wagner} established
the corresponding fixed-degree, large-shift hyperbolicity for
Dedekind zeta-functions.

Here we study the unshifted
Jensen polynomials 
\begin{equation}\label{eq:J-definition}
 J_d(X):=J^{d,0}(X)=\sum_{j=0}^{d}\binom dj\gamma(j)X^j.
\end{equation}
Jensen's criterion gives the equivalence
\begin{equation}\label{eq:RH-equivalence}
 \mathrm{RH}\quad\Longleftrightarrow\quad
 J_d\text{ is hyperbolic for every }d\ge1,
\end{equation}
as we recall in Proposition~\ref{prop:Jensen-criterion}. Hyperbolicity
is known through a large finite range of degrees. A criterion of
Chasse \cite{Chasse}, in the form stated in
\cite[Theorem~1.2]{GORTTW}, converts verification of RH for zeros
with $|\Im s|\le H$ into hyperbolicity for $d\le\lfloor H\rfloor^2$.
The computations of Platt and Trudgian \cite{PT} through height
$3\cdot10^{12}$ therefore give hyperbolicity for every
$d\le9\cdot10^{24}$. This finite verification does not control the unshifted family as $d\to\infty$.
The following theorem gives a quantitative and asymptotic form of the Jensen-P\'olya criterion for RH.

\begin{theorem}\label{thm:main}
There exists an absolute constant $C>0$ such that, for each integer $d\ge 2$, the polynomial $J_d(X)$ possesses at least $d - C\sqrt{d\log d}$ distinct negative real roots.
\end{theorem}

\begin{remark}
The proportion of real zeros therefore tends to one with an error
of size $O(\sqrt{\log d/d})$, confirming an asymptotic form of the
Jensen--P\'olya criterion for RH.
\end{remark}

We prove that the same phenomenon holds for the Dedekind zeta-function of every number
field. To state it, let $K$ be a number field of degree
$N=[K:\mathbb Q]=r_1+2r_2$, where $r_1$ is the number of real
embeddings and $r_2$ is the number of conjugate pairs of complex
embeddings. Write $\Delta_K>0$ for the absolute discriminant and
$\mathcal O_K$ for the ring of integers of $K$. The Dedekind
zeta-function is the meromorphic continuation of
\[
 \zeta_K(s)=\sum_{\mathfrak a\ne0}(\mathrm N\mathfrak a)^{-s}
          =\prod_{\mathfrak p}
             \bigl(1-(\mathrm N\mathfrak p)^{-s}\bigr)^{-1}
 \qquad(\Re s>1).
\]
Here the sum runs over the nonzero integral ideals of $\mathcal O_K$,
the product runs over its nonzero prime ideals, and
$\mathrm N\mathfrak a=|\mathcal O_K/\mathfrak a|$ is the absolute
norm. Define
\[
 \Gamma_{\mathbb R}(s)=\pi^{-s/2}\Gamma(s/2),\qquad
 \Gamma_{\mathbb C}(s)=2(2\pi)^{-s}\Gamma(s),
\]
and the entire completed function
\begin{equation}\label{eq:xi-K-definition}
 \xi_K(s)=\frac12s(s-1)\Delta_K^{s/2}
          \Gamma_{\mathbb R}(s)^{r_1}
          \Gamma_{\mathbb C}(s)^{r_2}\zeta_K(s).
\end{equation}
It satisfies $\xi_K(s)=\xi_K(1-s)$, and its zeros are precisely
the nontrivial zeros of $\zeta_K$ \cite[Section~2.1]{HSW}.
Its even expansion at the center defines
\begin{equation}\label{eq:FK-JK-definition}
 F_K(z)=\xi_K\!\left(\frac12+\sqrt z\right)
       =\sum_{j\ge0}\frac{\gamma_K(j)}{j!}z^j,
 \qquad
 J_{K,d}(X)=\sum_{j=0}^d\binom dj\gamma_K(j)X^j.
\end{equation}
The functional equation makes $F_K$ independent of the square root.
For $K=\mathbb Q$, these definitions recover $F$ and $J_d$.
The generalized Riemann hypothesis (GRH) for $\zeta_K$ asserts
that all of its nontrivial zeros have real part $1/2$.
By Proposition~\ref{prop:Jensen-criterion}, this is equivalent to
the assertion that every nonzero $J_{K,d}$ has only nonpositive
real zeros.
The following theorem is the general form of our result. 

\begin{theorem}\label{thm:Dedekind}
For each number field $K$, there exist constants $C_K>0$ and $d_0(K)\ge2$ such that for every integer $d\ge d_0(K)$, the polynomial $J_{K,d}(X)$ is nonzero and has at least $d-C_K\sqrt{d\log d}$ distinct negative real zeros. 
\end{theorem}

\begin{remark}
    The Jensen--P\'olya criterion for GRH requires real-rootedness
of every nonzero Jensen polynomial. Our theorem instead
shows that the proportion of real zeros tends to one,
while permitting a growing number of nonreal zeros.
Farmer~\cite{Farmer2022} has discussed the limitations
of interpreting Jensen polynomial hyperbolicity as
evidence for RH. Although the rescaled Jensen
polynomials converge locally uniformly to the
associated entire function, our estimate does not
rule out zeros of $\zeta_K(s)$ off the critical line.
In particular, it does not yield a new zero-free region
or zero-density estimate.
\end{remark}

We make no hypothesis on the central value $\zeta_K(1/2)$.
Such a value can vanish, as Armitage's work shows
\cite{Armitage}. Let $m_K\ge0$ denote the order of vanishing of
$F_K$ at zero, with $m_K=0$ when $F_K(0)\ne0$. Then $J_{K,d}$
is identically zero for $d<m_K$ and is nonzero for
$d\ge m_K$. This explains the sufficiently large degree formulation.
We will prove Theorem~\ref{thm:Dedekind}. Theorem~\ref{thm:main} is the $K=\mathbb{Q}$ case.

The proof uses the \emph{reciprocal polynomial}
\begin{equation}\label{eq:reciprocal-definition}
 R_{K,d}(x):=x^dJ_{K,d}(1/x)
           =\sum_{j=0}^d\binom dj\gamma_K(j)x^{d-j}.
\end{equation}

For nonzero arguments, this switch replaces each zero by its
reciprocal and preserves multiplicity, so it suffices to count
distinct negative real zeros of $R_{K,d}$. Multiplying the Taylor
series of $F_K(z)$ and $e^{xz}$ shows that the coefficient of $z^d$
in their product is $R_{K,d}(x)/d!$. Hence, Cauchy's coefficient formula
gives
\[
 R_{K,d}(x)=\frac{d!}{2\pi i}
       \oint_{|z|=\rho}F_K(z)e^{xz}\frac{\dd z}{z^{d+1}},
\]
where $\rho>0$ is arbitrary and the circle is traversed
counterclockwise. Following \cite{GORZ, GORTTW, Wagner}, we construct an expression to study the zeros of $J_{K,d}$. Earlier work with shifts $n$ reformulated the problem to admit Hermite-polynomial approximations. Here we take $n=0$ (no shift) and instead use a trigonometric reformulation.
Indeed, after a positive rescaling that balances the
growth of the gamma factors with the degree parameter, the main
contributions come from two conjugate saddle points, where the
derivative of the exponential phase vanishes. Rather nicely, these contributions
combine into a positive multiple of a cosine, whose sign crossings we estimate to count real zeros.

The point is to make this approximation uniform on an interval
long enough to capture many distinct real zeros. A concavity identity
controls the contour away from the saddles, while a uniform complex
Gaussian estimate evaluates their contributions. At the saddles,
the Dedekind zeta-function is evaluated in a fixed half-plane of
absolute convergence, which bounds its contribution to the phase.
An exact formula for the imaginary part of the exponent then yields
a net phase change of nearly $d\pi$, forcing alternating signs and
hence distinct real zeros. The endpoint losses account for the
error in Theorem~\ref{thm:Dedekind}, and hence capture the true difficulty of RH.

Section~\ref{sec:preliminaries} establishes the coefficient
properties and Jensen's criterion, and derives the integral
representation. We analyze the saddle points and their phase in
Section~\ref{sec:geometry} and bound the remaining factors in
Section~\ref{sec:amplitude}. These estimates give the uniform
oscillatory formula of Section~\ref{sec:oscillation}.
Section~\ref{sec:counting} completes the proof by counting sign
changes.

\section{The coefficients and the Cauchy integral}\label{sec:preliminaries}

We establish the coefficient and growth properties that justify
Jensen's criterion, and then express a rescaled reciprocal polynomial
by a Cauchy integral. Throughout the proof, fix the number field $K$
and write $F$, $\gamma(j)$, $J_d$, and $R_d$ for $F_K$,
$\gamma_K(j)$, $J_{K,d}$, and $R_{K,d}$, respectively.
The notation $U\ll V$ or $U=O(V)$ means $|U|\le CV$ for a
constant $C$, and $U\asymp V$ means the corresponding two-sided
bound between positive quantities. Constants may depend on $K$
unless they are explicitly declared absolute, and their values
may change between occurrences.

\subsection{Jensen's criterion}

Functional equations will supply the real entire function used in
Jensen's criterion. For the Riemann case, a positive theta kernel
also gives the sign of every coefficient.

\begin{lemma}\label{lem:positive-growth}
The function $F$ is a nonzero real entire function. With
$M_F(R)=\max_{|z|\le R}|F(z)|$, its order
$\limsup_{R\to\infty}\log\log M_F(R)/\log R$ is at most $1/2$.
In particular, $J_d$ is nonzero for all sufficiently large
$d$ and has degree at most $d$. When $K=\mathbb Q$, every
$\gamma(j)$ is strictly positive.
\end{lemma}

\begin{proof}
The function $\xi_K$ is entire of order one, is real on the real
axis, and is even about $1/2$. These standard analytic properties
are recalled in \cite[Section~2.1]{HSW}. Consequently, its even
Taylor expansion defines a real entire function $F$ of order at
most $1/2$. More explicitly, the bound
$\max_{|s|\le R}|\xi_K(s)|\le\exp(C_{K,\varepsilon}R^{1+\varepsilon})$
for every $\varepsilon>0$ gives
$\max_{|z|\le R}|F(z)|\le
\exp(C'_{K,\varepsilon}R^{(1+\varepsilon)/2})$ for large $R$.
The function is not identically zero because $\xi_K(s)>0$ for
real $s>1$. Its first nonzero Taylor coefficient occurs at some
finite index $m_K$, and its contribution to $J_d$ is nonzero for
every $d\ge m_K$.

For $K=\mathbb Q$, the theta-integral representation in
\cite[(1.1)--(1.2)]{LM}, with the normalization
\eqref{eq:xi-definition}, is
\begin{equation}\label{eq:theta-integral}
 F(z)=2\int_0^\infty\Phi(t)\cosh(t\sqrt z)\dd t,
\end{equation}
where
\begin{equation}\label{eq:Phi}
 \Phi(t)=\sum_{m=1}^{\infty}
 \left(4\pi^2m^4e^{9t/2}-6\pi m^2e^{5t/2}\right)
 e^{-\pi m^2e^{2t}}.
\end{equation}
Each summand is positive for $t\ge0$ because
$2\pi m^2e^{2t}>3$. Moreover,
$\Phi(t)\le C\exp(9t/2-\pi e^{2t})$, as follows by factoring
out $e^{-\pi e^{2t}}$ and bounding the remaining series by a
constant multiple of $\sum_{m\ge1}m^4e^{-\pi(m^2-1)}$.
On $|z|\le R$, the series for $\cosh(t\sqrt z)$ is bounded
absolutely by $e^{t\sqrt R}$. This integrable majorant justifies
termwise integration and gives
\begin{equation}\label{eq:gamma-integral}
 \gamma(j)=\frac{2j!}{(2j)!}
             \int_0^\infty\Phi(t)t^{2j}\dd t>0.
\end{equation}
\end{proof}

In the Dedekind case, Jensen's criterion must retain the sign of
the polynomial zeros. We use its this form for real entire
functions of order less than one 
\cite[Theorem~3.1 and Corollary~3.2]{OS}.

\begin{proposition}\label{prop:Jensen-criterion}
GRH for $\zeta_K$ holds if and only if every nonzero polynomial
$J_{K,d}$ has only nonpositive real zeros. For $K=\mathbb Q$,
this is equivalent to hyperbolicity of $J_d$ for every $d\ge1$.
\end{proposition}

\begin{proof}
The relation $\xi_K(1/2+w)=F(w^2)$ identifies GRH with the
assertion that all zeros of $F$ are nonpositive real numbers.
Under this assertion, the order bound gives the canonical product
\[
 F(z)=c z^{m_K}\prod_{\nu}\left(1+\frac z{\lambda_\nu}\right),
 \qquad c>0,\quad \lambda_\nu>0,\quad
 \sum_\nu\lambda_\nu^{-1}<\infty.
\]
Here zeros are repeated according to multiplicity, and positivity
of $F(z)$ for real $z>1/4$ fixes the sign of $c$.
Its coefficients are nonnegative. To apply the Jensen--P\'olya
criterion with a nonzero constant coefficient, put $H(z)=F(z)/z^{m_K}$.
For $d\ge m_K$, comparison of coefficients gives
\[
 J_d(X)=\frac{d!}{(d-m_K)!}\,X^{m_K}J_{H,d-m_K}(X),
\]
where $J_{H,k}$ is formed from the factorial Taylor coefficients
of $H$ by the same rule as in \eqref{eq:FK-JK-definition}.
The criterion applied to $H$ gives hyperbolicity of $J_{H,k}$,
so each nonzero $J_d$ is hyperbolic. Its nonnegative coefficients
exclude positive zeros.

Conversely, the polynomials $J_d(z/d)$ converge locally uniformly
to $F(z)$. Indeed, $\binom dj d^{-j}\le1/j!$, and this quantity
tends to $1/j!$ for each fixed $j$, so the absolutely convergent
Taylor series of $F$ gives a uniform majorant on every disk.
If every nonzero $J_d$ has only nonpositive real zeros, Hurwitz's
theorem shows that $F$ has no zero outside $(-\infty,0]$.
This proves the general equivalence. For $K=\mathbb Q$,
Lemma~\ref{lem:positive-growth} makes every coefficient
positive, so hyperbolicity already places every zero on the
negative real axis.
\end{proof}

\subsection{The gamma factors}

To make use of Cauchy's integral,
we balance the growth
of the gamma factors with the terms involving $d$. The discriminant
enters through the positive constant
\begin{equation}\label{eq:aK-definition}
 a_K=2\pi\Delta_K^{-1/N}.
\end{equation}
Let $T>a_Ke$ be the unique solution of
\begin{equation}\label{eq:T-scale}
 d=\frac{NT}{2\pi}\left(\log\frac{T}{a_K}-1\right),
\end{equation}
and put $\Lambda=2(\log(T/a_K)-1)$ and $B=T^2$.
Existence and uniqueness follow because the right side of
\eqref{eq:T-scale} increases strictly from zero to infinity on
this interval. These definitions give the exact identities
\begin{equation}\label{eq:T-identities}
 NT\Lambda=4\pi d,\qquad T=a_Ke^{\Lambda/2+1}.
\end{equation}
Taking logarithms yields
\[
 \Lambda+2\log\Lambda
 =2\log d+2\log(2/N)+\frac2N\log\Delta_K-2.
\]
Consequently, for the fixed field $K$, we have
\begin{equation}\label{eq:T-comparisons}
 \Lambda=2\log d+O_K(\log\log d),\qquad
 T\asymp_K\frac d{\log d}\qquad(d\to\infty).
\end{equation}
We define the rescaled reciprocal polynomial by
\begin{equation}\label{eq:A-definition}
 A_d(x)=\left(\frac Bd\right)^dR_d\!\left(\frac dB x\right)
       =\sum_{j=0}^d\binom dj\gamma(j)\left(\frac Bd\right)^j x^{d-j}.
\end{equation}
Positive scaling and reciprocation preserve the number of distinct
negative real zeros, so it suffices to count such zeros of $A_d$.
Cauchy's coefficient formula gives the representation
\begin{equation}\label{eq:Cauchy}
 A_d(x)=\frac{d!}{2\pi i\,d^d}
       \oint F(Bu)e^{dxu}\frac{\dd u}{u^{d+1}}
\end{equation}
on every positively oriented circle about zero. This follows
from the integral for $R_d$ by substituting $z=Bu$ and using
\eqref{eq:A-definition}.
For $q>0$, the corresponding arguments of $A_d$ and $J_d$ are
related by
\begin{equation}\label{eq:root-dictionary}
 A_d(-q)=0\quad\Longleftrightarrow\quad
 J_d\!\left(-\frac{B}{dq}\right)=0.
\end{equation}
Therefore, the real parameter $q$ locates negative
zeros of the Jensen polynomial $J_d$.

\section{Circles through the saddle points}\label{sec:geometry}

For $x=-q<0$, we choose a circle on which the large exponential
part of the integrand in \eqref{eq:Cauchy} has maximal modulus
at two conjugate saddles. We determine their positions, bound the
rest of the circle, and compute the
imaginary part of the exponent that governs the oscillation.
Throughout Sections~\ref{sec:geometry}--\ref{sec:counting},
we take $d$ above a fixed threshold depending on $K$, enlarged when
necessary, so that the estimates are uniform in the stated ranges.
All constants are fixed independently of $d$ and $q$. The field
remains fixed throughout, so the threshold may be enlarged to
absorb constants involving its degree or discriminant.

On the lower half-plane, we use the branches of the logarithm
and square root corresponding to $-\pi\le\arg u\le0$, with
boundary values taken from below, and use the conjugate branches
on the upper half-plane. The role of the scale is seen in the
exponent supplied by Stirling's formula. With $w=T\sqrt u$, the
identities \eqref{eq:T-identities} give
\begin{equation}\label{eq:scale-exponent}
 \frac{Nw}{2}\left(\Log\frac{w}{a_K}-1\right)
 =\frac{NT\sqrt u}{4}(\Lambda+\Log u)
 =\pi d\sqrt u\left(1+\frac{\Log u}{\Lambda}\right).
\end{equation}
Hence, the gamma and discriminant contributions have the same large parameter $d$ as
$e^{-dqu}u^{-d}$. Writing $\kappa=\pi/2$ to shorten the saddle
formulas, we combine these terms in the complex phase
\begin{equation}\label{eq:phase-definition}
 h_q(u)=-qu+2\kappa\sqrt u
             +\frac{2\kappa}{\Lambda}\sqrt u\Log u-\Log u.
\end{equation}
The exact factorization is proved in Section~\ref{sec:amplitude}.
Its exponential factor is $e^{dh_q(u)}$, so $\Re h_q$ controls
its size and $\Im h_q$ controls its oscillation. The saddle
equation is $h_q'(u)=0$, and we suppress the dependence on $d$
through $\Lambda$, which remains fixed when $q$ varies.

\subsection{Parameterization and the edge}

We describe the lower saddle using $v=\sqrt u=re^{-i\alpha}$,
so $r$ and $-\alpha$ are polar coordinates of $\sqrt u$, while
the saddle itself is $u=r^2e^{-2i\alpha}$. As $\alpha$ decreases
to zero, the conjugate saddles coalesce on the positive real axis
at a parameter $q_e$, called the \emph{edge}. The following lemma
parameterizes the branch $q>q_e$ on which the two contributions
produce oscillation.

\begin{lemma}\label{lem:saddle-parameters}
There are unique positive numbers $r_e$ and $q_e$ satisfying
\begin{equation}\label{eq:edge-definition}
 \kappa r_e\left(1+\frac{2\log r_e}{\Lambda}\right)=2,
 \qquad q_e=r_e^{-2}+\frac{2\kappa}{\Lambda r_e}.
\end{equation}
For every $q_e<q\le T$, there is a unique saddle of $h_q$ of the form
$u_q=r^2e^{-2i\alpha}$ with $r>0$ and $0<\alpha<\pi/2$.
Its parameters are characterized by
\begin{align}
 \frac{\Lambda\cos\alpha}{\kappa r}+\alpha\cot\alpha
  &=\frac\Lambda2+1+\log r,\label{eq:saddle-r}\\
 q&=r^{-2}+\frac{2\kappa}{\Lambda}
                    \frac{\alpha}{\sin\alpha}\,r^{-1}.
                    \label{eq:saddle-q}
\end{align}
As $\alpha$ increases, $r$ decreases strictly and $q$ increases
strictly. The following hold uniformly in this range:
\begin{equation}\label{eq:saddle-comparisons}
 1\le qr^2\le1+\frac{2\pi}{\Lambda},\qquad
 r\asymp q^{-1/2},\qquad \cos\alpha\asymp r.
\end{equation}
The edge satisfies $q_e=\pi^2/16+O(\Lambda^{-1})$.
For some absolute $t_0>0$, we also have
\begin{equation}\label{eq:edge-comparison}
 q-q_e\asymp\alpha^2\qquad(0<q-q_e\le t_0).
\end{equation}
\end{lemma}

\begin{proof}
Writing $v=\sqrt u$, the saddle equation $h_q'(u)=0$ becomes
\[
 qv+v^{-1}=\kappa\left(1+\frac{2+2\Log v}{\Lambda}\right).
\]
For $v=re^{-i\alpha}$, its imaginary and real parts give, respectively,
\begin{align}
 qr-r^{-1}&=\frac{2\kappa\alpha}{\Lambda\sin\alpha},
                       \label{eq:saddle-im}\\
 (qr+r^{-1})\cos\alpha
    &=\kappa\left(1+\frac{2+2\log r}{\Lambda}\right).
                       \label{eq:saddle-re}
\end{align}
Eliminating $q$ gives \eqref{eq:saddle-r} and
\eqref{eq:saddle-q}, and the elimination is reversible for
$0<\alpha<\pi/2$.

For fixed $0\le\alpha\le\pi/2$, let $E(r,\alpha)$ be the left side
of \eqref{eq:saddle-r} minus its right side. The continuous endpoint
values of $\alpha\cot\alpha$ are $1$ and $0$. We have
\[
 E_r=-\frac{\Lambda\cos\alpha}{\kappa r^2}-\frac1r<0,
\]
and $E$ tends to $+\infty$ and $-\infty$ as $r$ tends to zero
and infinity. There is therefore exactly one positive solution
$r(\alpha)$. For $0<\alpha<\pi/2$, we also have
\[
 E_\alpha=-\frac{\Lambda\sin\alpha}{\kappa r}
          +\frac{\sin\alpha\cos\alpha-\alpha}{\sin^2\alpha}<0.
\]
Implicit differentiation gives $r'=-E_\alpha/E_r<0$.
Since $\alpha/\sin\alpha$ is strictly increasing, both terms
in \eqref{eq:saddle-q} increase strictly with $\alpha$.

To check that this branch covers the required values of $q$, we
evaluate its endpoints. At $\alpha=0$ the equations give
\eqref{eq:edge-definition}. Direct differentiation of $h_q$ and
the edge equations also give
\[
 h_{q_e}'(r_e^2)=h_{q_e}''(r_e^2)=0,
\]
so the two saddles meet at a point where the quadratic term
vanishes. At the other endpoint we have
\[
 r(\pi/2)=e^{-\Lambda/2-1}=\frac{a_K}{T},
 \qquad q(\pi/2)>\frac{T^2}{a_K^2}.
\]
Since $T^2/a_K^2>T$ for sufficiently large $T$, this branch covers the
whole interval $(q_e,T]$.

The edge equation gives $1<r_e<2/\kappa=4/\pi$. Equation
\eqref{eq:saddle-im} consequently gives
\[
 qr^2=1+\frac{2\kappa r}{\Lambda}\frac{\alpha}{\sin\alpha}
 \le1+\frac{2\pi}{\Lambda}.
\]
It follows that $r\asymp q^{-1/2}$ and
\begin{equation}\label{eq:log-r-range}
 1+\frac{2\log r}{\Lambda}
   =1-\frac{\log q}{\Lambda}+O(\Lambda^{-2}).
\end{equation}
Since $q_e<q\le T$ and $\log T=\Lambda/2+1+\log a_K$, the
left-hand side lies between $1/3$ and $2$ for large $\Lambda$.
Substitution in \eqref{eq:saddle-re} then gives
$\cos\alpha\asymp r$. At the edge, the defining equation yields
$r_e=2/\kappa+O(\Lambda^{-1})$, which gives the asserted expansion
for $q_e$.

To obtain a uniform estimate in $\Lambda$ near the edge, we divide
$E$ by $\Lambda$ and treat $\delta=1/\Lambda$ as another
parameter. The resulting equation is analytic in
$(r,\alpha,\delta)$ near $(2/\kappa,0,0)$, is even in $\alpha$,
and has a nonzero derivative with respect to $r$ there. The analytic
implicit function theorem gives a fixed neighborhood
and uniformly bounded derivatives for all small $\delta\ge0$.
At $\delta=0$, the solution is $r(\alpha,0)=2\cos\alpha/\kappa$,
so \eqref{eq:saddle-q} gives $q(\alpha,0)=\kappa^2/(4\cos^2\alpha)$.
Analytic dependence and evenness in $\alpha$ therefore give
\begin{equation}\label{eq:edge-Taylor}
 q(\alpha,\delta)=q_e(\delta)
   +\left(\frac{\kappa^2}{4}+O(\delta)\right)\alpha^2
   +O(\alpha^4).
\end{equation}
The remainders are uniform, so the quadratic coefficient stays
positive and bounded away from zero. This proves
\eqref{eq:edge-comparison} on a sufficiently small fixed neighborhood.
\end{proof}

\subsection{The global circle estimate}

With the saddle parameters $r$ and $\alpha$ from
Lemma~\ref{lem:saddle-parameters}, write the lower semicircle as
$u=r^2e^{-2i\beta}$, where $0\le\beta\le\pi/2$.
Concavity in the coordinate $\cos\beta$ gives the following uniform
estimate for the decrease of $\Re h_q$ away from the saddle.

\begin{lemma}\label{lem:circle-gap}
There is an absolute constant $c>0$ such that, uniformly for
$q_e<q\le T$ and $0\le\beta\le\pi/2$, we have
\begin{equation}\label{eq:global-gap}
 c(\cos\beta-\cos\alpha)^2
 \le\Re h_q(u_q)-\Re h_q(r^2e^{-2i\beta}).
\end{equation}
In particular, the saddle is the unique maximum of the real part
of the phase on the lower semicircle.
\end{lemma}

\begin{proof}
Set $z=\cos\beta$ and let $G(z)$ be the real part of
$h_q(r^2e^{-2i\beta})$. Substitution in \eqref{eq:phase-definition}
gives
\[
 G(z)=-qr^2(2z^2-1)
 +2\kappa r\left(1+\frac{2\log r}{\Lambda}\right)z
 -\frac{4\kappa r}{\Lambda}\arccos z\sqrt{1-z^2}-2\log r.
\]
For $0<z<1$, differentiation yields
\[
 G''(z)=-4qr^2+\frac{4\kappa r}{\Lambda}
              \frac{\beta-\sin\beta\cos\beta}{\sin^3\beta},
 \qquad z=\cos\beta.
\]
The quotient is nonnegative and extends continuously to
$0\le\beta\le\pi/2$, with limiting value $2/3$ at zero.
It is therefore bounded independently of the saddle parameters.
Since $qr^2=1+O(\Lambda^{-1})$ and $r$ is bounded, we obtain
$G''(z)\le-c$ for an absolute constant $c>0$. At $z=\cos\alpha$, the first derivative
vanishes because $h_q'(u_q)=0$. Integrating this inequality twice
proves \eqref{eq:global-gap}, with the endpoints included by
continuity.
\end{proof}

For the local approximation, let $\theta_q=-2\alpha$ and
$D_q=u_q^2h_q''(u_q)$. We have the angular expansion
\[
 h_q(r^2e^{i(\theta_q+t)})=h_q(u_q)-\frac12D_qt^2+O(t^3)
 \qquad(t\to0).
\]
The next lemma controls this quadratic coefficient and the remainder
uniformly in $q$.

\begin{lemma}\label{lem:Hessian}
The quadratic coefficient is given by
\begin{equation}\label{eq:D-definition}
 D_q=u_q^2h_q''(u_q)
    =1-\frac{\kappa re^{-i\alpha}}2
                    \left(1+\frac{2\log r-2i\alpha}{\Lambda}\right).
\end{equation}
Its real part and modulus satisfy the uniform estimates
\begin{equation}\label{eq:D-bounds}
 \Re D_q\asymp\sin^2\alpha,
 \qquad |D_q|\asymp\sin\alpha,
 \qquad \Re D_q>0
\end{equation}
for $q_e<q\le T$. All angular derivatives of $h_q(r^2e^{i\theta})$ of orders one
through three are uniformly bounded for $-\pi\le\theta\le0$.
\end{lemma}

\begin{proof}
Differentiation of \eqref{eq:phase-definition} gives
$u^2h_q''(u)=1-(\kappa/2)\sqrt u(1+\Log u/\Lambda)$,
which proves \eqref{eq:D-definition}. We now use the saddle
equations \eqref{eq:saddle-im}--\eqref{eq:saddle-re} to write
its real and imaginary parts in the form
\begin{align}
 \Re D_q&=\sin^2\alpha+
  \frac{\kappa r}{\Lambda}
  \left(\cos\alpha-\frac{\alpha\cos(2\alpha)}{\sin\alpha}\right),
                           \label{eq:D-real}\\
 \Im D_q&=\frac{\kappa r}{2}
  \left[\left(1+\frac{2\log r}{\Lambda}\right)\sin\alpha
                  +\frac{2\alpha\cos\alpha}{\Lambda}\right].
                           \label{eq:D-imag}
\end{align}
The parenthesis in \eqref{eq:D-real} is nonnegative and
$O(\sin^2\alpha)$. Indeed, after division by $\sin^2\alpha$ it is
\[
 \frac{\sin\alpha\cos\alpha-\alpha\cos(2\alpha)}{\sin^3\alpha},
\]
which extends continuously to the closed interval, with value $4/3$
at zero. Its numerator vanishes at zero and has derivative
$2\alpha\sin(2\alpha)\ge0$. This proves the bounds for the real
part, while \eqref{eq:D-imag} gives $|\Im D_q|\le C\sin\alpha$.
Near $\alpha=0$, the quantities $r$
and $1+2\log r/\Lambda$ are bounded below by positive constants,
which gives $\Im D_q\ge c\sin\alpha$. Away from zero the real
part supplies the lower bound for $|D_q|$, completing the proof
of \eqref{eq:D-bounds}.

Up to an additive constant, the angular phase is a linear
combination of $e^{i\theta}$, $e^{i\theta/2}$,
$\theta e^{i\theta/2}$, and $\theta$. Its coefficients are bounded
because $qr^2$, $r$, and $r|\log r|/\Lambda$ are bounded by
Lemma~\ref{lem:saddle-parameters}. These functions and their first
three derivatives are bounded on $[-\pi,0]$, proving the remaining
assertion.
\end{proof}

To keep the local interval inside the semicircle, we take its
angular half-width to be
\begin{equation}\label{eq:ell-definition}
 \ell_q=\epsilon_0\min\{\alpha,\cos\alpha\},
\end{equation}
for a sufficiently small fixed $\epsilon_0>0$. We rescale this
interval to $[-1,1]$ and normalize the quadratic coefficient to
have modulus one. Here $\|f\|_{C^k}$ denotes the maximum of the
sup norms of the derivatives of orders zero through $k$.

\begin{corollary}\label{cor:local-geometry}
The interval $|\theta-\theta_q|\le\ell_q$ lies in $(-\pi,0)$.
On the rest of the lower semicircle, we have
\begin{equation}\label{eq:discard-gap}
 \Re h_q(r^2e^{i\theta})
 \le\Re h_q(u_q)-c_g\min\{(q-q_e)^2,q^{-1}\},
\end{equation}
where $c_g>0$ is absolute after $\epsilon_0$ is fixed. On $[-1,1]$, put
\begin{equation}\label{eq:p-definition}
 p_q(y)=\frac{h_q(r^2e^{i(\theta_q+\ell_qy)})-h_q(u_q)}
                   {|D_q|\ell_q^2}.
\end{equation}
These functions have uniformly bounded $C^3$ norms and satisfy
\begin{equation}\label{eq:p-properties}
 p_q(0)=p_q'(0)=0,\quad
 p_q''(0)=-\frac{D_q}{|D_q|},\quad
 \Re p_q\le0,\quad |p_q'(y)|\ge\frac{|y|}{2}.
\end{equation}
\end{corollary}

\begin{proof}
The distances from $\theta_q$ to $0$ and $-\pi$ are $2\alpha$
and $2(\pi/2-\alpha)$, respectively. Since
$\cos\alpha\le\pi/2-\alpha$, the local interval stays inside the
semicircle. At its endpoints, the corresponding change in $\beta$
has size $\ell_q/2$. The mean value theorem and
$\ell_q\le\epsilon_0\alpha$ give
\[
 |\cos\beta-\cos\alpha|\ge c\ell_q\sin\alpha.
\]
Monotonicity of cosine gives the same bound outside the local
interval. Lemma~\ref{lem:circle-gap} therefore gives a gap of at
least $c\ell_q^2\sin^2\alpha$.
Near the edge, this gives a lower bound of order
$\alpha^4\asymp(q-q_e)^2$.
For $q\ge q_1$, where $q_1$ is a sufficiently large fixed absolute
constant, it gives a lower bound of order
$\cos^2\alpha\asymp q^{-1}$.
On the intervening interval $q_e+t_0\le q\le q_1$, monotonicity and
\eqref{eq:edge-comparison} give $\alpha\ge c\sqrt{t_0}$.
The bounds $r\ge q_1^{-1/2}$ and $\cos\alpha\asymp r$ also imply
$\cos\alpha\ge c q_1^{-1/2}$. Thus $\ell_q$ and $\sin\alpha$ are
bounded below by positive absolute constants on this interval.
This proves \eqref{eq:discard-gap}.

The value and first two derivatives of $p_q$ at zero follow from
the angular expansion at the saddle, while the maximality in
Lemma~\ref{lem:circle-gap} gives $\Re p_q\le0$.
To control the remaining derivatives, use Lemma~\ref{lem:Hessian}
and $\ell_q/|D_q|\le C\epsilon_0\alpha/\sin\alpha\le C\epsilon_0$
to obtain
\[
 \sup_{[-1,1]}|p_q'''|\le C\epsilon_0.
\]
Since $|p_q''(0)|=1$, Taylor's theorem then gives
$|p_q'(y)|\ge|y|-C\epsilon_0|y|^2\ge|y|/2$ when $\epsilon_0$ is
sufficiently small. Integrating from zero also bounds the lower
derivatives, proving the uniform $C^3$ estimate.
\end{proof}

\subsection{The phase at the saddle}

The oscillation is governed by $S(q):=\Im h_q(u_q)$.
We compute this function exactly so that its change over a growing
interval can be estimated from the endpoints.

\begin{lemma}\label{lem:action}
The function $S$ satisfies
\begin{equation}\label{eq:exact-action}
 S(q)=2\alpha-\sin(2\alpha)
      +\frac{4\kappa r}{\Lambda}
                       (\sin\alpha-\alpha\cos\alpha).
\end{equation}
It extends continuously with value zero at $q_e$, and its derivative satisfies
\begin{equation}\label{eq:action-derivative}
 S'(q)=r^2\sin(2\alpha)>0.
\end{equation}
For fixed absolute $t_0>0$ sufficiently small and $q_1$ sufficiently
large, we have
\begin{align}
 0\le S(q_e+t)&\le Ct^{3/2} &&(0<t\le t_0),
                                    \label{eq:action-edge}\\
 S(q)&\ge\pi-Cq^{-1/2} &&(q_1\le q\le T).
                                    \label{eq:action-tail}
\end{align}
\end{lemma}

\begin{proof}
Substituting the saddle into \eqref{eq:phase-definition} gives
\begin{align*}
 S(q)={}&qr^2\sin(2\alpha)
 -2\kappa r\left(1+\frac{2\log r}{\Lambda}\right)\sin\alpha
 -\frac{4\kappa r\alpha}{\Lambda}\cos\alpha+2\alpha.
\end{align*}
Equation \eqref{eq:saddle-re} implies
\[
 \kappa r\left(1+\frac{2\log r}{\Lambda}\right)
 =(qr^2+1)\cos\alpha-\frac{2\kappa r}{\Lambda}.
\]
Substituting this identity into the preceding display gives
\eqref{eq:exact-action}. To differentiate $S$, we use the fact
that $h_q'(u_q)=0$, obtaining
\[
 S'(q)=\Im\left(-u_q+h_q'(u_q)u_q'\right)
       =\Im(-u_q)=r^2\sin(2\alpha).
\]
Both $2\alpha-\sin(2\alpha)$ and
$\sin\alpha-\alpha\cos\alpha$ are nonnegative and
$O(\alpha^3)$ near zero. Lemma~\ref{lem:saddle-parameters} proves
\eqref{eq:action-edge} and the continuous value at the edge.

For the other endpoint, put $\nu=\pi/2-\alpha$ and drop the
nonnegative last term of \eqref{eq:exact-action} to obtain
\[
 S(q)\ge 2\alpha-\sin(2\alpha)
      =\pi-2\nu-\sin(2\nu)\ge\pi-4\nu.
\]
Since $\nu\asymp\cos\alpha\asymp r\asymp q^{-1/2}$ for large
$q$, this gives \eqref{eq:action-tail}.
\end{proof}

\section{The amplitude in the Cauchy integral}\label{sec:amplitude}

Stirling's formula separates the integrand into $e^{dh_q}$ and a
remaining analytic factor, the \emph{amplitude}. Away from
the saddles, the decrease of the exponential will dominate the
polynomial growth allowed for this amplitude. Near the saddles,
we use the stronger bounds supplied by the Euler product in
$\Re s\ge4$. The following lemma isolates the gamma remainder
and bounds it using the uniform complex form of Stirling's
formula \cite[Section~5.11]{DLMF}.

\begin{lemma}\label{lem:Stirling-amplitude}
Put
\begin{equation}\label{eq:amplitude-constants}
 \mu_K=2-\frac{r_1}{4},\qquad
 c_K=2^{3r_1/4+r_2-1}\pi^{r_1/4}\Delta_K^{1/4}.
\end{equation}
There are $W_0>0$ and a holomorphic nonvanishing function $g_K$
on an open neighborhood of
$\{w:|\arg w|\le\pi/2,\ |w|\ge W_0\}$ for which
\begin{equation}\label{eq:Stirling-identity}
 \xi_K\!\left(\frac12+w\right)
 =c_Kw^{\mu_K}
    \exp\!\left\{\frac{Nw}{2}
                  \left(\Log\frac{w}{a_K}-1\right)\right\}
       \zeta_K\!\left(\frac12+w\right)g_K(w).
\end{equation}
All powers use the principal logarithm. A single holomorphic
logarithm of $g_K$ satisfies
\begin{equation}\label{eq:g-bounds}
 \log g_K(w)=O_K(|w|^{-1}),\qquad
 (\log g_K)'(w)=O_K(|w|^{-2})
 \qquad(|\arg w|\le\pi/2).
\end{equation}
\end{lemma}

\begin{proof}
Separate the two gamma remainders by setting
\[
 g_{\mathbb R}(w)=
 \frac{\Gamma(w/2+1/4)}
 {\sqrt{2\pi}(w/2)^{w/2-1/4}e^{-w/2}},\qquad
 g_{\mathbb C}(w)=
 \frac{\Gamma(w+1/2)}{\sqrt{2\pi}w^we^{-w}}.
\]
Then the identity is exact with
\[
 g_K(w)=\left(1-\frac1{4w^2}\right)
        g_{\mathbb R}(w)^{r_1}g_{\mathbb C}(w)^{r_2}.
\]
Indeed, substitution in \eqref{eq:xi-K-definition} gives the
power $w^{2-r_1/4}$ and the positive constant $c_K$ above.
The exponential terms combine as
\[
 \frac{Nw}{2}\left(\Log\frac{w}{2\pi}-1\right)
       +\frac w2\log\Delta_K
 =\frac{Nw}{2}\left(\Log\frac{w}{a_K}-1\right).
\]
Uniform Stirling asymptotics give
$g_{\mathbb R}(w)=1+O(|w|^{-1})$ and
$g_{\mathbb C}(w)=1+O(|w|^{-1})$ in the larger sector
$|\arg w|\le2\pi/3$. Thus $g_K(w)=1+O_K(|w|^{-1})$ there.
Choose $W_0$ large enough that $|g_K(w)-1|<1/2$ in this sector.
The convergent logarithm about $1$ gives a single holomorphic
logarithm of size $O_K(|w|^{-1})$. For $|\arg w|\le\pi/2$,
a disk of radius a sufficiently small fixed multiple of $|w|$
stays in the larger sector and away from zero. Cauchy's derivative
estimate on this disk proves the second bound in \eqref{eq:g-bounds}.
\end{proof}

We restrict $q\le c_0T$, where $0<c_0\le1$ will be chosen small
enough for the local estimates below.
On the lower semicircle, define
\begin{equation}\label{eq:a-definition}
 w_q(\theta)=Tr e^{i\theta/2},\qquad
 a_q(\theta)=c_Kw_q(\theta)^{\mu_K}
     \zeta_K\!\left(\frac12+w_q(\theta)\right)g_K(w_q(\theta))
\end{equation}
for $-\pi\le\theta\le0$. Since
$Tr\gg T/\sqrt{c_0T}\to\infty$, the preceding lemma applies on
the whole semicircle. The identity \eqref{eq:scale-exponent}
then gives the exact factorization
\begin{equation}\label{eq:integrand-factorization}
 F(T^2u)e^{-dqu}u^{-d}=a_q(\theta)e^{dh_q(u)},
 \qquad u=r^2e^{i\theta}.
\end{equation}
Since $d$ is an integer, the factor $e^{-d\Log u}$ agrees with
$u^{-d}$ on either logarithmic branch. The following lemma
bounds the amplitude on the whole semicircle and its variation
across the local saddle interval.

\begin{lemma}\label{lem:amplitude-bounds}
For $c_0$ sufficiently small, the amplitude satisfies
\begin{equation}\label{eq:a-at-saddle}
 |a_q(\theta_q)|\asymp_K(Tr)^{\mu_K},\qquad
 \sup_{-\pi\le\theta\le0}
       \frac{|a_q(\theta)|}{|a_q(\theta_q)|}\le C_KT^N
\end{equation}
uniformly for $q_e<q\le c_0T$. Moreover, the normalized local amplitude
\begin{equation}\label{eq:f-definition}
 f_q(y)=\frac{a_q(\theta_q+\ell_qy)}{a_q(\theta_q)},
 \qquad -1\le y\le1,
\end{equation}
satisfies $f_q(0)=1$ and $\|f_q\|_{C^1}\le C$.
\end{lemma}

\begin{proof}

We first record a polynomial bound on the whole semicircle:
\begin{equation}\label{eq:global-zeta}
 |\zeta_K(s)|\le C_K|s|^N
 \qquad(\Re s\ge1/2,\ |s|\ge2).
\end{equation}
This is weaker than the standard convexity bound
\cite[Proposition~3.8]{HSW}. For completeness, it also follows
directly from the functional equation and Stirling's formula.
On $\Re s=-1/2$ they give
$\zeta_K(s)\ll_K(1+|s|)^N$, while the Euler product bounds
$\zeta_K(s)$ on $\Re s=3/2$. Apply the Phragm\'en--Lindel\"of
principle in the intervening strip to the holomorphic function
\[
 \frac{(s-1)\zeta_K(s)}{(s+2)^{N+1}}.
\]
The order-one bound for $\xi_K$ and the reciprocal gamma factors
give finite-order growth in this strip. The displayed function is
bounded on both boundary lines, so the principle bounds it
throughout the strip. On $\Re s\ge1/2$,
$|s|\ge2$, we have $|s-1|\ge|s|/2$ and $|s+2|\le2|s|$.
This proves \eqref{eq:global-zeta} up to $\Re s=3/2$.
For larger real parts, the Euler product gives the same bound.

For the local bounds, observe that the angular displacement of
$w_q$ on $|\theta-\theta_q|\le\ell_q$ is at most
$\ell_q/2\le\epsilon_0\cos\alpha/2$. The cosine changes by at most
this displacement, so the real part of $w_q$ stays a fixed positive
proportion of its value at the saddle. Consequently, we have
\begin{equation}\label{eq:local-real-w}
 \Re w_q(\theta)\ge cTr\cos\alpha\ge cT/q.
\end{equation}
Take $c_0$ small enough that this implies
$\Re(1/2+w_q(\theta))\ge4$ on every local interval.
The absolutely convergent Euler product then bounds $|\zeta_K|$
above and below by positive constants through the inequalities
\[
 \zeta(\sigma)^{-N}\le |\zeta_K(\sigma+it)|
       \le\zeta(\sigma)^N\qquad(\sigma>1),
\]
because at most $N$ prime ideals lie above a rational prime and
each has norm at least that prime. The lower bound also follows
by bounding the Euler product for $1/\zeta_K$.
Together with $g_K=1+O_K((Tr)^{-1})$, this proves the first assertion
of \eqref{eq:a-at-saddle}.

On the whole lower semicircle, we have $\Re(1/2+w_q(\theta))\ge1/2$,
$|w_q(\theta)|=Tr$, and the modulus of $w_q(\theta)^{\mu_K}$ is
independent of $\theta$. Since $r\le r_e<4/\pi$, applying
\eqref{eq:global-zeta} and dividing by the saddle estimate proves
the second assertion of \eqref{eq:a-at-saddle}.

For the local derivative bound, we return to $\Re s\ge4$, where
absolute convergence gives
\[
 \left|\frac{\zeta_K'(s)}{\zeta_K(s)}\right|
 \le C_Ke^{-c\Re s}\qquad(\Re s\ge4).
\]
Indeed, logarithmic differentiation of the Euler product gives
coefficients bounded by $N$ times those of $-\zeta'/\zeta$,
because the sum of the residue degrees above a rational prime
is at most $N$.
For $\sigma=\Re s\ge4$, their absolute sum is at most
$N2^{-\sigma/2}\sum_{n\ge2}(\log n)n^{-2}$, which proves
the estimate. On the scaled interval, we also have
\[
 \left|\frac{d}{dy}w_q(\theta_q+\ell_qy)\right|
 =\frac12\ell_qTr\le C\frac Tq.
\]
Logarithmic differentiation of \eqref{eq:a-definition}, using
\eqref{eq:g-bounds} and \eqref{eq:local-real-w}, therefore gives
\[
 \left|\frac{d}{dy}\log a_q(\theta_q+\ell_qy)\right|
 \le C\left(1+\frac Tq e^{-cT/q}\right)\le C.
\]
The logarithm exists locally because every factor is nonzero.
Together with the bound $|f_q(y)|\asymp1$ already supplied by these
factors, the logarithmic derivative estimate gives
$|f_q'(y)|\le C$ and proves the $C^1$ assertion.
\end{proof}

\section{A uniform oscillatory formula}\label{sec:oscillation}

We combine the circle and amplitude estimates with a uniform
Gaussian approximation. The two conjugate saddle contributions
then give a real oscillatory formula for $A_d(-q)$.

\subsection{A complex Gaussian estimate}

Near the edge, the ratio $\Re D_q/|D_q|$ tends to zero by
\eqref{eq:D-bounds}, so the quadratic term becomes predominantly
oscillatory. An estimate based only on the absolute value of the
Gaussian loses uniformity in this regime. We instead normalize
by $|D_q|$ and use integration by parts to retain the cancellation
away from the saddle. The following lemma makes this argument
uniform even when the normalized quadratic coefficient approaches
the imaginary axis.

\begin{lemma}\label{lem:Gaussian}
Let $p\in C^3([-1,1];\C)$ and $f\in C^1([-1,1];\C)$ satisfy
\[
 p(0)=p'(0)=0,\qquad p''(0)=-A,\qquad
 |A|=1,\qquad \Re A>0,\qquad \Re p(y)\le0.
\]
Assume the bounds
\[
 |p'(y)|\ge c|y|,\qquad
 \|p\|_{C^3}+\|f\|_{C^1}\le C
\]
for fixed positive constants $c,C$. Then, for every real $n\ge1$, we have
\begin{equation}\label{eq:Gaussian-estimate}
 \int_{-1}^1f(y)e^{np(y)}\dd y
 =f(0)\sqrt{\frac{2\pi}{nA}}+O(n^{-3/5})
\end{equation}
uniformly in these functions and in $A$.
The square root is principal. When $f(0)=1$, the relative error is
$O(n^{-1/10})$.
\end{lemma}

\begin{proof}
We compare the integral with a Gaussian on
$[-\eta_n,\eta_n]$, where $\eta_n=n^{-2/5}$, and estimate the
remaining intervals by integration by parts. This cutoff balances
the cubic approximation error $n\eta_n^4$ with the tail bound
$(n\eta_n)^{-1}$, both of which have size $n^{-3/5}$.
Taylor's theorem gives $p(y)=-Ay^2/2+O(|y|^3)$.
To use this approximation uniformly, note that if $z_1,z_2$ have nonpositive real
parts, the identity
\[
 e^{z_1}-e^{z_2}=(z_1-z_2)
                \int_0^1e^{tz_1+(1-t)z_2}\dd t
\]
implies $|e^{z_1}-e^{z_2}|\le|z_1-z_2|$.
Applying this inequality with $z_1=np(y)$ and $z_2=-nAy^2/2$
shows that replacing the exponential on $[-\eta_n,\eta_n]$ changes the
integral by
\[
 O\!\left(n\int_{-\eta_n}^{\eta_n}|y|^3\dd y\right)=O(n^{-3/5}).
\]
Replacing $f(y)$ by $f(0)$ on the same interval contributes only
$O(\eta_n^2)=O(n^{-4/5})$.

To bound the remaining part on $[\eta_n,1]$, we integrate once by parts:
\[
 \int_{\eta_n}^1f e^{np}\dd y
 =\left[\frac{f e^{np}}{np'}\right]_{\eta_n}^1
 -\frac1n\int_{\eta_n}^1
       \left(\frac{f'}{p'}-\frac{fp''}{(p')^2}\right)e^{np}\dd y.
\]
The boundary terms are $O((n\eta_n)^{-1})$. Since $|e^{np}|\le1$,
the remaining integral term has modulus at most
\[
 \frac Cn\int_{\eta_n}^1(y^{-1}+y^{-2})\dd y
 =O((n\eta_n)^{-1})=O(n^{-3/5}).
\]
The same argument applies to $[-1,-\eta_n]$.

For $\Re A>0$, the complex Gaussian formula gives
$\int_\R e^{-nAy^2/2}\dd y=\sqrt{2\pi/(nA)}$.
This identity follows from the real Gaussian formula by analytic
continuation in the right half-plane. Its tail from $\eta_n$ to infinity
is uniformly $O((n\eta_n)^{-1})$: integration by parts using
$(e^{-nAy^2/2})'=-nAy e^{-nAy^2/2}$ gives a boundary term of
modulus at most $(n\eta_n)^{-1}$ and an integral bounded by
$n^{-1}\int_{\eta_n}^\infty y^{-2}\dd y$. The boundary at infinity
vanishes because the real part of $A$ is positive, and the same
bound holds for the negative tail.
Combining these bounds with the central approximation proves
\eqref{eq:Gaussian-estimate}, and division by the leading modulus
$\sqrt{2\pi/n}$ when $f(0)=1$ gives the relative estimate.
\end{proof}

\subsection{The two saddle contributions}

The local coordinate $\theta=\theta_q+\ell_qy$ gives the normalized
integrand $f_q(y)e^{n_qp_q(y)}$, where we define
\begin{equation}\label{eq:nq-definition}
 n_q=d|D_q|\ell_q^2.
\end{equation}
This parameter measures the size of the quadratic term across the
local interval and controls the error in the following approximation.

\begin{proposition}\label{prop:lower-contribution}
Let $L_d(q)$ be the contribution of the lower semicircle to
\eqref{eq:Cauchy} with $x=-q$ and radius $r^2$. Define
\begin{equation}\label{eq:Z-definition}
 Z_d(q)=\frac{d!}{d^d}\,
        \frac{a_q(\theta_q)e^{dh_q(u_q)}}{\sqrt{2\pi dD_q}},
\end{equation}
where the square root is principal. The leading term $Z_d(q)$ is
nonzero, and the lower semicircle integral satisfies
\begin{equation}\label{eq:lower-estimate}
 L_d(q)=Z_d(q)
 \left\{1+O(n_q^{-1/10})+
 O_K\!\left(d^{N+1}e^{-c_gd\min\{(q-q_e)^2,q^{-1}\}}\right)\right\}
\end{equation}
uniformly for $q_e<q\le c_0T$ with $n_q\ge1$.
If $q_e+\tau\le q\le Q\le c_0T$, with $\tau>0$ sufficiently small
and $Q$ sufficiently large, then we have
\begin{equation}\label{eq:nq-lower}
 n_q\ge c\min\{d\tau^{3/2},d/Q\}.
\end{equation}
\end{proposition}

\begin{proof}
Parameterizing the lower semicircle by $u=r^2e^{i\theta}$,
$-\pi\le\theta\le0$, gives $\dd u/u=i\dd\theta$.
Equation \eqref{eq:integrand-factorization} therefore gives
\[
 L_d(q)=\frac{d!}{2\pi d^d}
        \int_{-\pi}^0a_q(\theta)e^{dh_q(r^2e^{i\theta})}\dd\theta.
\]
On $|\theta-\theta_q|\le\ell_q$, substitute
$\theta=\theta_q+\ell_qy$. The integrand becomes
\[
 a_q(\theta_q)e^{dh_q(u_q)}f_q(y)e^{n_qp_q(y)}.
\]
Corollary~\ref{cor:local-geometry} and
Lemma~\ref{lem:amplitude-bounds} verify every hypothesis of
Lemma~\ref{lem:Gaussian}, with $A=D_q/|D_q|$.
After multiplying by $\ell_qd!/(2\pi d^d)$, its leading term is
exactly $Z_d(q)$, and its relative error is $O(n_q^{-1/10})$.

The remainder of the lower semicircle has length at most $\pi$ in
the angular variable. By \eqref{eq:discard-gap} and
\eqref{eq:a-at-saddle}, its absolute contribution is at most
\[
 C_K\frac{d!}{d^d}|a_q(\theta_q)|T^Ne^{d\Re h_q(u_q)}
             e^{-c_gd\min\{(q-q_e)^2,q^{-1}\}}.
\]
Dividing by $|Z_d(q)|$ gives the bound
\[
 C_KT^N\sqrt{d|D_q|}\,
        e^{-c_gd\min\{(q-q_e)^2,q^{-1}\}}
 \le C_Kd^{N+1}e^{-c_gd\min\{(q-q_e)^2,q^{-1}\}},
\]
because $|D_q|\le C$ and $T\le d$ for all large $d$.
This proves \eqref{eq:lower-estimate}.

Near the edge, the bounds $|D_q|\asymp\alpha$ and
$\ell_q\asymp\alpha$ give $n_q\asymp d(q-q_e)^{3/2}$.
For $q$ above a fixed large constant, the bounds $|D_q|\asymp1$
and $\ell_q\asymp q^{-1/2}$ give $n_q\asymp d/q$.
On the intervening fixed interval, the positive lower bounds in
the proof of Corollary~\ref{cor:local-geometry} give $n_q\asymp d$.
Combining these ranges proves \eqref{eq:nq-lower}.
\end{proof}

We now choose an interval on which both errors tend to zero
uniformly. Fix $\epsilon_0$ as in
Corollary~\ref{cor:local-geometry}, and hence fix $c_g$.
Choose $M>0$ with $c_gM^2>N+4$. Decrease $c_0>0$, if
necessary, so that Lemma~\ref{lem:amplitude-bounds} holds and
$c_gN/(4\pi c_0)>N+4$. Put
\begin{equation}\label{eq:window}
 \tau=M\sqrt{\frac{\log d}{d}},\qquad Q=c_0T,
 \qquad I_d=[q_e+\tau,Q].
\end{equation}
For sufficiently large $d$ this interval is nonempty, and the bound
$\Lambda\ge\log d$ and the chosen constants give
\[
 c_gd\tau^2>(N+4)\log d,
 \qquad
 \frac{c_gd}{Q}=\frac{c_gN\Lambda}{4\pi c_0}>(N+4)\log d.
\]
Together with $n_q\gg\log d$ from \eqref{eq:nq-lower}, these
inequalities make both errors in \eqref{eq:lower-estimate} tend
to zero uniformly. At $q=Q$, the effective parameter satisfies
$n_Q\asymp d/Q\asymp\log d$, which accounts for the logarithmic
error below. We do not optimize this rate, since an error of
size $o(|Z_d(q)|)$, uniform on $I_d$, suffices to count sign
changes.
Adding the conjugate upper contribution gives the oscillatory
formula we need.

\begin{corollary}\label{cor:oscillatory}
Uniformly for $q\in I_d$, we have
\begin{equation}\label{eq:oscillatory}
 A_d(-q)=2\Re Z_d(q)+O\!\left(|Z_d(q)|(\log d)^{-1/10}\right).
\end{equation}
\end{corollary}

\begin{proof}
On $I_d$, the exponential error in \eqref{eq:lower-estimate}
is $O(d^{-3})$, and $n_q\gg\log d$ bounds the remaining relative
error as claimed. The upper semicircle contributes exactly
$\overline{L_d(q)}$ because $F(\overline u)=\overline{F(u)}$
and $q$ is real. The Cauchy integrand is single-valued, while
\eqref{eq:integrand-factorization} uses conjugate branches on
the two semicircles. Adding their contributions proves
\eqref{eq:oscillatory}.
\end{proof}

\section{Counting the real zeros}\label{sec:counting}

The oscillatory formula determines the sign of $A_d(-q)$ when the
phase reaches an integer multiple of $\pi$. To count these points,
we show that the phase differs from $dS(q)$ by a bounded continuous
correction and then use the endpoint estimates for $S$.

\begin{lemma}\label{lem:phase-correction}
On $I_d$, we can write $Z_d(q)=|Z_d(q)|e^{i\vartheta_d(q)}$,
where $\vartheta_d$ is a continuous real function satisfying
\begin{equation}\label{eq:phase-correction}
 \vartheta_d(q)=dS(q)+\omega_d(q),\qquad |\omega_d(q)|\le C.
\end{equation}
\end{lemma}

\begin{proof}
To choose the phase continuously, use the Euler-product logarithm
\[
 \Log\zeta_K(s)=\sum_{\mathfrak p}\sum_{m\ge1}
               \frac{(\mathrm N\mathfrak p)^{-ms}}m
 \qquad(\Re s>1),
\]
together with the logarithm of $g_K$ in
Lemma~\ref{lem:Stirling-amplitude} and the principal argument
of $D_q$. The positive real prefactor in
\eqref{eq:Z-definition} does not contribute to the phase, so with
$v_q=re^{-i\alpha}$ these choices give
\begin{equation}\label{eq:omega-formula}
 \omega_d(q)=-\mu_K\alpha
  +\Im\Log\zeta_K\!\left(\frac12+Tv_q\right)
  +\Im\log g_K(Tv_q)-\frac12\Arg D_q.
\end{equation}
These branch choices make every term continuous. Since
$0<\alpha<\pi/2$ and $\Re D_q>0$, both the first term and
$\Arg D_q$ are bounded. At the saddle arguments, the Euler-product
logarithm has absolute value at most $N\log\zeta(4)$ by
\eqref{eq:local-real-w} and the choice of $c_0$.
The remaining term is $O((Tr)^{-1})$ by \eqref{eq:g-bounds},
which proves the uniform bound in \eqref{eq:phase-correction}.
\end{proof}

The bounded correction allows us to count zeros from the net phase
change, without assuming that the corrected phase is monotone.

\begin{proof}[Proof of Theorem~\ref{thm:Dedekind}]
For all sufficiently large $d$, Corollary~\ref{cor:oscillatory} and
Lemma~\ref{lem:phase-correction} give
\[
 A_d(-q)=2|Z_d(q)|
  \left(\cos\vartheta_d(q)+O((\log d)^{-1/10})\right)
 \qquad(q\in I_d).
\]
The normalized error is then uniformly smaller than $1/2$.
By Lemma~\ref{lem:action}, the phase change satisfies
\begin{align*}
 \vartheta_d(Q)-\vartheta_d(q_e+\tau)
 &=d\bigl(S(Q)-S(q_e+\tau)\bigr)+O(1)\\
 &\ge \pi d-Cd\bigl(Q^{-1/2}+\tau^{3/2}\bigr)-C.
\end{align*}
Let $\Delta_d$ denote this positive phase change. There are at
least $\Delta_d/\pi-1$ integer multiples of $\pi$ between the
endpoint values. We can choose points of $I_d$ in increasing order
at which these consecutive values are attained: after each chosen
point, the intermediate value theorem on the remaining interval
supplies the next one. At these points the cosine equals $1$ and
$-1$ alternately, and the error bound preserves the signs of
$A_d(-q)$. Each intervening open interval therefore contains a
distinct zero. There are at least $\Delta_d/\pi-2$ such intervals,
which gives at least
\begin{equation}\label{eq:root-count-before-scale}
 d-Cd\bigl(Q^{-1/2}+\tau^{3/2}\bigr)-C
\end{equation}
distinct real zeros of $q\mapsto A_d(-q)$ in $I_d$.

The choices in \eqref{eq:window} and
\eqref{eq:T-comparisons} give
\[
 dQ^{-1/2}=O(\sqrt{d\log d}),\qquad
 d\tau^{3/2}=O\!\left(d^{1/4}(\log d)^{3/4}\right)
             =O(\sqrt{d\log d}).
\]
Thus $A_d$ has at least $d-C\sqrt{d\log d}$ distinct negative
real zeros. The correspondence \eqref{eq:root-dictionary} gives the same
number of distinct negative zeros of $J_d$. For all sufficiently
large $d$, Lemma~\ref{lem:positive-growth} ensures that $J_d$ is
nonzero and has degree at most $d$. It consequently has at most
$C\sqrt{d\log d}$ nonreal zeros, counted with multiplicity.
Since $K$ was arbitrary and all constants are independent of $d$,
this proves the theorem.
\end{proof}

\begin{remark}
The argument also locates the real zeros that it counts. Under
\eqref{eq:root-dictionary}, the interval $I_d$ corresponds to
\[
 \left[-\frac{T^2}{d(q_e+\tau)},-\frac{T^2}{dQ}\right].
\]
Since $q_e+\tau\asymp1$, $Q=c_0T$, and $T\asymp d/\log d$,
there are constants $C_1,c_1>0$, depending on $K$, such that this interval
is contained in
\[
 \left[-C_1\frac{d}{(\log d)^2},-\frac{c_1}{\log d}\right]
\]
for all sufficiently large $d$. This locates the
$d-C\sqrt{d\log d}$ distinct zeros of $J_d$ in this range.
\end{remark}

\section*{Appendix: AxiomProver Lean certificate and AI declaration}
\label{app:lean-certificate}

Theorem~\ref{thm:Dedekind} has been formalized in Lean~4 using
AxiomProver. 
This work also establishes
the Riemann case of Theorem~\ref{thm:main} for every $d\ge2$.
The formal definitions use the normalization of the completed
Dedekind zeta-function in \eqref{eq:xi-K-definition} and the
Jensen polynomials in \eqref{eq:FK-JK-definition}.
The Dedekind theorem assumes only that $K$ is a number field.
It assumes neither GRH nor the nonvanishing of
$\zeta_K(1/2)$. The formal development supplies the analytic
and geometric inputs required for the proof, including the
completed Dedekind zeta-function and its functional equation,
the Gamma and zeta estimates, the saddle-point analysis,
and the final sign-change and root-counting argument.

We use Lean~4.33.1 and
\href{https://github.com/leanprover-community/mathlib4}{Mathlib},
the Lean mathematical library, at revision
\href{https://github.com/leanprover-community/mathlib4/commit/0df444a360eaa60ab8c11dca51a86af692955474}
{\texttt{0df444a360ea}}.
It builds on the existing Riemann Jensen-polynomial
formalization in the same repository.
The completed Dedekind-zeta construction includes a port of
a formal proof from
\href{https://github.com/anthropics/fermats-last-theorem}
{\texttt{anthropics/fermats-last-theorem}}
at revision
\href{https://github.com/anthropics/fermats-last-theorem/tree/6e837e75355538c7f80bab5b956861e86c4eacc2}
{\texttt{6e837e753555}}.
The relevant proof source and its supporting utility are
included in this repository and compiled as part of the
development. The file \texttt{lake-manifest.json} records
the pinned revisions of the software dependencies.
The formal statements, proof sources, dependency blueprint,
and instructions for checking the proofs are available at
\begin{center}
\url{https://github.com/AxiomMath/RiemannJensen}.
\end{center}

This paper was written through human--AI collaboration with
AxiomProver. The project arose from the author's goal of
extending his earlier joint works on shifted Jensen polynomials
\cite{GORZ,GORTTW} to the unshifted setting.
The author developed the general strategy for the mathematical
proof through renormalization and the Cauchy integral formula, where Hermite polynomials in the shifted setting are replaced by $\cos$. 
AxiomProver carried out the detailed calculations, optimized
parameter choice and interval selection to complete the proof. Finally, AxiomProver
produced the Lean formalization. 
The author reviewed the argument and is accountable for the
mathematical statements, exposition, and conclusions presented
here. ChatGPT assisted with editing the manuscript.

\end{document}